\documentclass[]{article}
\usepackage{amssymb,amsmath,amsthm}

\usepackage{pdfsync}

\newcommand{\R}{\mathbf{R}}

\newcommand {\E}{\mathrm{E}}

\renewcommand{\d}{\text{\rm d}}

\newcommand{\sL}{\mathcal{L}}
\newcommand{\sP}{\mathcal{P}}

\newtheorem{stat}{Statement}[section]
\newtheorem{proposition}[stat]{Proposition}
\newtheorem{corollary}[stat]{Corollary}
\newtheorem{theorem}[stat]{Theorem}
\newtheorem{lemma}[stat]{Lemma}
\newtheorem{Condition}[stat]{Condition}

\theoremstyle{definition}

\numberwithin{equation}{section}

\begin{document}

\title{\bf On non-existence of global solutions to a class of stochastic heat equations.}%
	
\author{Mohammud Foondun \and Rana D. Parshad}

%
\maketitle
\begin{abstract}
	We consider nonlinear parabolic
	SPDEs of the form 
	$\partial_t u=-(-\Delta)^{\alpha/2} u + b(u) +\sigma(u)\dot w$, where
	$\dot w$ denotes space-time white noise. The functions 
	 $b$ and $\sigma$ are both locally Lipschitz continuous. Under some suitable conditions on the parameters of these SPDEs, we show that the second moment of their solutions blow up in finite time.  This complements recent works of Khoshnevisan and his coauthors; see for instance \cite{FK},\cite {FK1} and \cite{CJK} as well as those of Chow \cite{Chow1} and \cite{Chow2}.  Furthermore, upon comparing our stochastic equations with their deterministic counterparts, we find that our results indicates that the presence of noise might affect the occurrence of blow-up.


		
		
	\noindent{\it Running Title:} Non-existence and 
		parabolic SPDEs.
\end{abstract}

\section{Introduction}

In \cite{FK}, \cite{FK1} and \cite{CJK}, Khoshnevisan and his coauthors initiated a research program, where they systematicaly studied the large time behavior of stochastic equations of the following type
\begin{equation}\label{eq0}\begin{split}
\partial_t u(t,\,x)&=\sL u(t,\,x)+\sigma(u(t,\,x))\dot{F}(t,\,x)\\
u(0\,,x)&=u_0(x),\\
\end{split}
\end{equation}
where $\sL$ is the generator of a L\'evy process, $\sigma$ is a globally Lipschitz function, $\dot{F}$ is a noisy perturbation, and $u_0(x)$ is specified initial data satisfying certain suitable conditions. Among other things, the authors studied the {\it upper p-th moment Liapounov exponent} $\bar{\gamma}(p)$ defined as

\begin{equation*}
 \bar{\gamma}(p):=\limsup_{t\rightarrow \infty} \frac{1}{t}\ln \E(|u(t,\,x)|^p)\quad\text{for all}\quad p\in(0,\infty).
\end{equation*}

 A motivation for this study comes from the Parabolic Anderson model and the phenomenon of intermittency which this model exhibits. Instead of going into details about this phenomenon, we refer the reader to \cite{CarmonaMolchanov:94} for details and references.  One of the main results in the above works, is the existence and strict positivity of those upper Liapounov exponents.  A common assumption in all the works mentioned above, is that  the operator $\sL$ models regular diffusion so one considers say the Laplace operator. Another important assumption is that the multiplicative non-linearity $\sigma$ is assumed globally Lipschitz. Indeed, this is required to ensure global existence of solutions. In this short note, we consider a subclass of \eqref{eq0}, where we  consider a fractional Laplacian operator instead of the diffusion process, and we drop the global Lipschitz property of $\sigma$. More precisely we look at the following equation:

\begin{equation}\label{eq1}\begin{split}
\partial_t u(t,\,x)&=-(-\Delta)^{\alpha/2} u(t,\,x)+b(u(t,\,x))+\sigma(u(t,\,x))\dot{w}(t,\,x)\\
u(0\,,x)&=u_0(x),\\
\end{split}
\end{equation}
where $\{\dot{w}(t,\,x) \}_{t\geq0, x\in \R}$ denotes space-time white noise. $-(-\Delta)^{\alpha/2}$ is the fractional Laplacian; the $L^2$-generator of a symmetric stable process of order $\alpha$, so that $\E \exp(i\xi \cdot X_t)=\exp(-t|\xi|^\alpha)$. We will follow Walsh \cite{Walsh} and interpret the above as an It\^o-type stochastic PDE. The functions $b$ and $\sigma$ are assumed to be positive and locally Lipschitz. A standard localisation argument show that the above equation has a unique random field solution with finite moments, provided that 
\begin{equation*}
\Upsilon(\beta):=\frac{1}{2\pi}\int_{-\infty}^\infty\frac{\d \xi}{\beta+2|\xi|^\alpha} < \infty\quad \text{for all}\quad \beta>0;
\end{equation*}
confer with \cite{FK} and \cite{Dalang}. This implies that $\alpha\in(1,2]$, a condition which will be enforced throughout this paper. We will also assume the initial data $u_0$ is a positive bounded measurable function satisfying the following lower bound:
\begin{equation}\label{ini}
\inf_{x\in R}u_0(x)\geq \kappa,
\end{equation}
where $\kappa$ is a positive constant.
\par
There are many physical motivations to consider a fractional Laplacian operator. These operators appear in many models in non-Newtonian fluids, in models of viscoelasticity such as Kelvin-Voigt models, various heat transfer processes in fractal and disordered media and models of fluid flow and acoustic propogation in porous media. Interestingly, they have also recently been applied to pricing derivative securities in financial markets. See \cite{Lakes2009,La69,B10} for details on a number of the afforementioned applications. The essential feature behind many of these models, is that they aim to capture the nonlinear relation between stress and strain, in non-Newtonian or viscoelastic models perse. Or perhaps, aim to model the disparity between scaling laws associated with diffusion, in euclidian geometries, and fractal geometries, such as certain porous media. Thus there are a plethora of reasons why one might consider investigating the asymptotic behavior of equations of the type \eqref{eq1}. In particular, such investigations might help us glean information in many real physical scenarios, involving sub or super diffusions.
\par
The primary contribution of the current manuscript is to show that under certain conditions on the initial data and parameters, the second moment of the solution to \eqref{eq1}, blows up in finite time, and hence we have non-existence of the upper Liapounov exponents.  When we say that a function $f$ {\it blows up in finite time}, we mean that there exists a time $T$ such that for all $t\geq T$, $f(t)=\infty$. And when we say that a statement holds whenever a quantity $Q$ is {\it large enough}, we mean that the statement holds whenever $Q\geq Q_0$, where $Q_0$ is some positive constant.

 Our result will hold under at least one of the following two conditions which are non-linear growth conditions on $\sigma$ and $b$ respectively.
\begin{Condition}\label{A}
There exist constants $\beta$ and $K_1>0$, both strictly positive such that 
\begin{equation*}
\inf_{x\in R}\frac{\sigma(x)}{|x|^{1+\beta}}\geq K_1.
\end{equation*}
\end{Condition}
 
\begin{Condition}\label{B}
There exist constants $\gamma$ and $K_2$ both strictly positive such that 
\begin{equation*}
\inf_{x\in R}\frac{b(x)}{|x|^{1+\gamma}}\geq K_2.
\end{equation*}
\end{Condition}

Denote $\{u(t,\,x)\}_{t\geq 0,\,x\in \R}$ as the solution to \eqref{eq1} so that $\{\E |u(t,\,x)|^2\}_{t\geq 0,\,x\in \R}$ denotes the corresponding second moment.
 
 \begin{theorem}\label{theorem1}
Suppose $\{u(t,\,x)\}_{t\geq 0,\,x\in \R}$ denote the unique solution to \eqref{eq1}. We then have the following 
\begin{itemize}
\item[(a)] Under Condition \ref{A}, $\{\E |u(t,\,x)|^2\}_{t\geq 0,\,x\in \R}$ the second moment of the solution,  blows up in finite time whenever $\kappa$ is large enough.
\item[(b)] Under Condition \ref{B}, $\{\E |u(t,\,x)|^2\}_{t\geq 0,\,x\in \R}$ blows up in finite time for any initial data satisfying \eqref{ini}, whenever $\gamma <\alpha$. But if $\gamma \geq \alpha$, finite time blow up occurs only whenever $\kappa$ is large enough.
 \end{itemize}

 \end{theorem}

We remark that blow-up occurs pointwise, that is for each $x\in \R$. This is in contrast with the results in \cite{Chow1} and \cite{Chow2}.

Blow-up problems in deterministic cases, is an area of intense study, starting with the classical work of Fujita \cite{Fu}. Blow up can mean the solutions themselves tends to infinity in finite time or sometimes their derivatives blow up in finite time; for instance solutions to the Burgers' equation, starting with smooth initial data, can experience loss of smoothness in finite time \cite{W74}, that is $\lim_{t\rightarrow T^{\ast}<\infty}||\nabla u||_2^2=+\infty$. The review articles of \cite{levine} and \cite{Deng-Levine} contain a detailed exposition of the subject.   A common approach in tackling these various  {\it blow up} problems is to  consider a ``right" quantity, which blows up in finite time, usually by exercising a differential inequality. The quantity at hand is usually linked to a certain ``norm'' of the solution.  We will use a similar approach to prove our main result. One of the main difficulties here is to find this ``right" quantity; see Proposition \ref{averageBU}.

We now briefly give an outline of the paper. In Section 2, we recall some facts about the heat kernel of the fractional Laplacian and prove some important estimates. In Section 3, we prove the main result. Finally in section 4, we provide some concluding remarks.  Throughout the paper, we use the letter $c$ with or without subscripts to denote a constant whose value is not important and may vary from places to places. And by $f(x)\asymp g(x)$, we mean that there exists a constant $K$ such that $\frac{1}{K}g(x)\leq f(x)\leq K g(x)$.

\section{Some auxiliary results.}
We begin this section by developing some required background. Let $p(t,\,\cdot)$ denote the transition density function of a symmetric stable process of order $\alpha$. The existence of this density is a well known fact; see for instance \cite{Sato}.  Set
\begin{equation}
\label{eqp}
 \sP_t f(x):=\int_{-\infty}^\infty p(t,\,x-y)f(y)\,\d y.
\end{equation}

We can now follow the method of Walsh \cite{Walsh} to see that that \eqref{eq1} admits a mild solution $u$ if and only if $u$ is a predictable process that satisfies
\begin{equation}\label{Soln}\begin{split}
 u(t,\,x)&=\sP_tu_0(x)+\int_{-\infty}^\infty \int_0^tb(u(s,\,y))p(t,\,y-x)\,\d s\,\d y\\
 &+\int_{-\infty}^\infty \int_0^t\sigma(u(s,\,y))p(t,\,y-x)\,w(\d s\,\d y).
\end{split}
\end{equation}
Here the proof of local existence follows via standard localisation means.
The proof of our theorem will involve some bounds on the transition density function appearing in \eqref{eqp}. Our first proposition recalls a few well known properties which we will use to derive those necessary bounds. We do not present proofs but mention \cite{Sato}, where the reader can find lots of details about symmetric stable processes.

\begin{proposition}\label{density}
The transition density $p(t,\,\cdot)$ of a strictly $\alpha$-stable process, satisfies the following 
\begin{itemize}
\item[(a)] \begin{equation}\label{scaling2}p(st,\,x)=t^{-1/\alpha}p(s, t^{-1/\alpha} x ),\end{equation}
\item[(b)] For $t$ large enough such that $p(t,\,0)\leq 1$ and $a>2$, we have 
\begin{equation}\label{monotone2}
p(t,\,(x-y)/a)\geq p(t,\,x)p(t,\,x)\quad \text{for}\quad \text{all}\quad x\in \R,
\end{equation}
\item[(c)] \begin{equation}\label{asymp}
            p(t,\,x)\asymp t^{-1/\alpha}\wedge \frac{t}{|x|^{1+\alpha}}
           \end{equation}
\end{itemize}
\end{proposition}

As a first application of the above proposition, we have the following lemma. We will use this to derive a very useful ODE.

\begin{lemma}\label{Holder}
 Let $0\leq s\leq c_1 t$, where $c_1$ is a strictly positive constant. Then the following inequality holds
 \begin{equation*}
 \int_\R p^2(s,\,x)p^2(t,\,x) \d x\geq c_2t^{-3/\alpha},
 \end{equation*}
 where $c_2$ is a positive constant depending on $\alpha$ and $c_1$.
 \end{lemma}

 \begin{proof}
 H\"older's inequality with respect to the probability measure $p(s,\,x)\d x$ gives
 \begin{equation*}
 \int_\R p(t,\,x)p(s,\,x)\d x\leq \left( \int_\R p^2(t,\,x)p(s,\,x)\d x\right)^{1/2}.
 \end{equation*}
Another application of H\"older's inequality yields
 \begin{eqnarray*}
 \int_\R p^2(t,\,x)p(s,\,x)\d x&\leq& \left( \int_\R p^2(t,\,x)p^2(s,\,x)\d x\right)^{1/2}\left( \int_\R p^2(t,\,x)\d x\right)^{1/2}\\
 &=& \| p(t,\,\cdot)\|_{L^2(\R)}\left( \int_\R p^2(t,\,x)p^2(s,\,x)\d x\right)^{1/2},
 \end{eqnarray*}
where $\|p(t,\,\cdot)\|_{L^2(\R)}^2:= \int_\R p^2(t,\,x)\d x$.  
 Combining the above two inequalities, we obtain
 \begin{eqnarray*}
 \int_\R p^2(t,\,x)p^2(s,\,x)\d x\geq \frac{\left(  \int_\R p(t,\,x)p(s,\,x)\d x\right)^4}{\| p(t,\,\cdot)\|^2_{L^2(\R)}}.
 \end{eqnarray*}
Note that
 \begin{equation*}
 \int_\R p(t,\,x)p(s,\,x)\d x= p(t+s,\, 0),
 \end{equation*}
 and 
 \begin{equation*}
 \| p(t,\cdot)\|^2_{L^2(\R)}=p(2t,\,0).
 \end{equation*}
 Using the fact that $p(t,\,0)\asymp t^{-1/\alpha}$, we obtain the required inequality.
 \end{proof}

Our next result concerns the blow-up of solutions to the class of ODE mentioned above.  We will use this to conclude the proof of our main result.

\begin{proposition}\label{ODE}
Consider the following ordinary differential equation.
\begin{equation*}
  y'(s)s^b=y(s)^{1+a}
\end{equation*}
where $a,\,b>0$ with initial condition $y(t_0)$. If $b\leq1$, the solution blows up for any positive initial condition $y(t_0)$. However, if $b>1$, then the solution blows up whenever $y(t_0)$ is large enough.
\end{proposition}

\begin{proof}We begin by looking at the simplest case when $b=1$. Elementary calculus show that the solution $y(t)$ satisfies
\begin{equation*}
 y(t)^{-a}=y(t_0)^{-a}-a\ln (\frac{t}{t_0}).
\end{equation*}
The above shows that $y(t)$ blows up in finite time for any  positive initial condition $y(t_0)$. We now turn to the case when $b\not=1$.  Again some calculus shows that $y(t)$ satisfies 
\begin{equation*}
 y(t)^{-a}=y(t_0)^{-a}+\frac{a\left(t^{1-b}-t_0^{1-b}\right)}{b-1}.
\end{equation*}
In the case that $b>1$, the above implies that $y(t)$ blows up only when 
\begin{equation}\label{ini-cond}
 y(t_0)^a>\frac{t_0^{b-1}(b-1)}{a}.
\end{equation}
When $b<1$, the solution blows up for any positive initial condition.
\end{proof}

We end this section by recalling Jensen's inequality which we will use frequently in the forthcoming section.
\begin{lemma}\label{Jensen}
Let $p(\d x)$ be a probability measure on $\R$ and $u$ be a non-negative function. Then the following holds whenever $f$ is a convex function,
\begin{equation*}
f\left(\int_R u(x)p(\d x) \right)\leq \int_Rf(u(x)) p(\d x).
\end{equation*}
\end{lemma}

\section{Proof of the main result.}

The proof of the main theorem hinges on the next proposition. We first fix a couple of notations.

Denote
 \begin{equation}\label{F}
 F(t):= \int_\R \E |u(t,\,y)|^2 p^2(t,\,y)\d y,
 \end{equation}
 and 
 
 \begin{equation}\label{G}
 G(t):= \int_\R \E |u(t,\,y)| p(t,\,y)\d y,
 \end{equation}
 where $\{u(t,\,x)\}_{t\geq 0,\,x\in \R}$ is the solution to \eqref{eq1}.  We can then state the following proposition. Sugitani  \cite{Su} was the main source of inspiration for its proof.
 
 \begin{proposition}\label{averageBU} The following two statements hold.
 \begin{itemize}
 \item[(a)]  $F(t)$ blows up in finite time if and only if for all $x\in \R$, $\E |u(t,\,x)|^2$ blows up in finite time,
 \item[(b)]  $G(t)$ blows up in finite time if and only if for all $x\in \R$, $\E |u(t,\,x)|$ blows up in finite time.
 \end{itemize}
    
 \end{proposition}
 
 \begin{proof}
 We will only prove part (a) of the proposition. The proof of the second part is similar. Note that it is enough to show that $\E |u(t,\,x)|^2=\infty$ whenever $F(t)=\infty$ for some large $t$.  The other implication is easy.  
Using property \eqref{scaling2}, we can write
 \begin{eqnarray*}
 p(t-s,\,x-y)=\left(\frac{s}{t-s}\right)^{1/\alpha} p(s,\,\left(\frac{s}{t-s}\right)^{1/\alpha}(x-y))
 \end{eqnarray*}
Let $\tilde{c}_\alpha:=\frac{1}{2^\alpha+1}$. We can now use \eqref{monotone2} to bound the right hand side of the above display as follows.
 \begin{equation}\label{bound-p}
 p(t-s,\,x-y)\geq\left(\frac{s}{t-s}\right)^{1/\alpha}p(s,\,x)p(s,\,y),
 \end{equation}
for $t_0 \leq s\leq \tilde{c}_\alpha t$, where $t_0$ is chosen large enough so that $p(s,\,0)\leq 1$ for all $s\geq t_0$.
 From the mild formulation of the solution i.e \eqref{Soln} and the positivity of $b$ and the initial condition $u_0(x)$, we obtain
 \begin{equation}\label{L_2-mild}
 \E |u(t,\,x))|^2\geq\int_0^t \int_\R\E |\sigma(u(s,\,y))|^2 p^2(t-s,\,x-y)\d y\,\d s.
 \end{equation}
Since the integrand appearing in the right hand side of the above display is strictly positive, we have 
\begin{equation*}
 \E |u(t,\,x))|^2\geq \int_{t_0}^{\tilde{c}_\alpha t} \int_\R\E |\sigma(u(s,\,y))|^2 p^2(t-s,\,x-y)\d y\,\d s. 
\end{equation*}

 Using the bound \eqref{bound-p}, and the lower bound on $\sigma$ together with Jensen's inequality, we can write
 \begin{equation}\label{lower}\begin{split}
 \E |u(t,\,x))|^2&\geq \int_{t_0}^{\tilde{c}_\alpha t} \left(\frac{s}{t-s}\right)^{2/\alpha}p^2(s,\,x)\int_\R\E |\sigma(u(s,\,y))|^2 p^2(s,\,y)\d y\,\d s\\
 &\geq c_1\int_{t_0}^{\tilde{c}_\alpha t} \left(\frac{s}{t-s}\right)^{2/\alpha}p^2(s,\,x)\int_\R(\E |u(s,\,y)|^2)^{1+\beta} p^2(s,\,y)\d y\,\d s\\
 &\geq c_1\int_{t_0}^{\tilde{c}_\alpha t} \left(\frac{s}{t-s}\right)^{2/\alpha}p^2(s,\,x)\|p(t,\,\cdot)\|^2_{L^2(\R)}\\
 &\times \int_\R(\E |u(s,\,y)|^2)^{1+\beta} \frac{p^2(s,\,y)}{\|p(s,\,\cdot)\|^2_{L^2(\R)}}\d y\,\d s\\
&\geq c_1\int_{t_0}^{\tilde{c}_\alpha t} \left(\frac{s}{t-s}\right)^{2/\alpha}p^2(s,\,x)\|p(s,\,\cdot)\|^2_{L^2(\R)}\\
 &\times \left(\int_\R \E |u(s,\,y)|^2 \frac{p^2(s,\,y)}{\|p(s,\,\cdot)\|^2_{L^2(\R)}}\d y\right)^{1+\beta}\,\d s.
\end{split}
 \end{equation}
 The result now follows from the above inequality and the strict positivity of the transition function.
 \end{proof}
 
 We are now ready to prove our main result, Theorem \ref{theorem1}.
\begin{proof}
We proceed in two steps. The first step will be carried out assuming condition \ref{A}, while the second step will be completed under condition \ref{B}.

{\it Step 1:}
In view of Proposition \ref{averageBU}, we aim to show  that $F(s)$ blows up in finite time. We start with \eqref{Soln}. Bearing in mind that $b$ is positive, we end up with
 \begin{equation*}\begin{split}
\int_{\R}\E |u(t\,,x)|^2p^2(t,\,x)\d x&\geq \int_{\R}\left|\int_\R u_0(y)p(t,\,x-y)\d y\right|^2 p^2(t,\,x)\d x\\
&+\int_{\R} \int_\R\int_0^t\E|\sigma(u(s,\,y))|^2p^2(t-s,\,x-y)\d s \d y  p^2(t,\,x)\d x\\
&:=I_1+I_2.
\end{split}
 \end{equation*}
We bound the $I_1$ first as this is straightforward. The lower bound on the initial condition yields
 \begin{equation*}
 \int_\R u_0(y)p(t,\,x-y)\d y\geq \kappa.
 \end{equation*}
 Hence, we have 
 \begin{equation*}
 I_1\geq \kappa^2 \|p(t,\cdot)\|^2_{L^2(\R)}.
 \end{equation*}
 
 To bound $I_2$, we will use some of the estimates from the proof of Proposition \ref{averageBU}. We first use inequality \eqref{lower} to obtain 
 \begin{equation*}\begin{split}
 I_2 &\geq c_1\int_{t_0}^{\tilde{c}_\alpha t} \int_{\R}\left(\frac{s}{t-s}\right)^{2/\alpha}p^2(s,\,x)p^2(t,\,x)\|p(s,\,\cdot)\|^2_{L^2(\R)}\d x\\
 &\times \left(\int_\R \E |u(s,\,y)|^2 \frac{p^2(s,\,y)}{\|p(s,\,\cdot)\|^2_{L^2(\R)}}\d y\right)^{1+\beta}\,\d s.
 \end{split}
 \end{equation*}
 We now use the fact that $\|p(s,\,\cdot)\|^2_{L^2(\R)}\asymp s^{-1/\alpha}$ and the lower bound obtained in Lemma \ref{Holder} to write
 
\begin{eqnarray*}
I_2&\geq& c_1\int_{t_0}^{\tilde{c}_\alpha t}\left(\frac{s}{t-s}\right)^{2/\alpha}t^{-3/\alpha}s^{\beta/\alpha}\left(\int_{\R}\E|u(s,\,y)|^2p^2(s,\,y)\d y \right)^{1+\beta}\d s\\
&\geq& c_1\int_{t_0}^{\tilde{c}_\alpha t}\frac{s^{(2+\beta)/\alpha}}{t^{5/\alpha}} \left(\int_{\R}\E|u(s,\,y)|^2p^2(s,\,y)\d y \right)^{1+\beta}\d s\\
\end{eqnarray*}

 We now combine the estimates for $I_1$ and $I_2$ to obtain
 \begin{equation*}\begin{split}
& \int_{\R}\E|u(s,\,y)|^2p^2(s,\,y)\d y\\
&\qquad \geq c_2\int_{t_0}^{\tilde{c}_\alpha t}\frac{s^{(2+\beta)/\alpha}}{t^{5/\alpha}} \left(\int_{\R}\E|u(s,\,y))|^2p^2(s,\,y)\d y \right)^{1+\beta}\d s\\
&\qquad+\kappa^2 c_3 t^{-1/\alpha},\\
\end{split}
\end{equation*} 
 
 The definition of $F$ and some algebra yield
\begin{equation*}
 F(t)t^{5/\alpha}\geq \kappa^2 c_3 t^{4/\alpha}+c_2\int_{t_0}^{\tilde{c}_\alpha t}s^{(2+\beta)/\alpha} F(s)^{1+\beta}\d s,.
\end{equation*}
We set $Y(t):=F(t)t^{5/\alpha}$ to obtain 
\begin{equation*}
 Y(t)\geq \kappa^2 c_3 t^{4/\alpha}+c_2\int_{t_0}^{\tilde{c}_\alpha t}\frac{Y(s)^{1+\beta}}{s^{(3+4\beta)/\alpha}}\d s.
\end{equation*}

We are interested in the blow-up of $Y(\cdot)$, by a comparison principle; see for instance \cite{smo}, it suffices to consider the following ordinary differential equations.
\begin{eqnarray*}
 Y'(t)t^{(3+4\beta)/\alpha}=c_4Y(t)^{1+\beta}\quad \text {for}\quad t\geq t_0,
\end{eqnarray*}
with initial condition $Y(t_0)=c_3\kappa^2 t_0^{4/\alpha}$.  Note that $3+4\beta>\alpha$, so we can use Proposition \ref{ODE} to conclude that the above equation blows up in final time whenever $\kappa$ is large enough.  Hence $F$ blows up in finite time as well. An application of Proposition \ref{averageBU} concludes the first part of the proof.

{\it Step 2:}  As in the previous step, we start with the mild formulation of the solution but take expectation to obtain

\begin{equation*}
 \E u(t,\,x)=\int_{\R}u_0(y)p(t,\,x-y)\, \d y+ \int_0^t\int_{\R} \E b(u(s,\,y) p(t-s,\,x-y)\,\d y\,\d s
\end{equation*}

We aim to show that the function $G(s)$ defined by \eqref{G} blows up in finite time. Proposition \ref{averageBU} and an application of H\"older's inequality will then yield the result.  We begin by multiplying the above display by $p(t,\,x)$ and integrate over $\R$ to obtain
\begin{eqnarray*}
 G(s)&=&\int_{\R}\int_{\R}u_0(y)p(t,\,x-y)\, \d y p(t,\,x)\, \d x\\
&+&\int_{\R}\int_0^t\int_{\R} \E b(u(s,\,y)) p(t-s,\,x-y)\,\d y\,\d s\,p(t,\,x) \d x\\
&:=&I_3+I_4.
\end{eqnarray*}
$I_3$ is easily seen to be bounded below by $\kappa$.  We now bound $I_4$. As in Proposition \ref{averageBU}, there exists a $t_0>0$ large enough so that for $t_0\leq s\leq \tilde{c}_\alpha t$,  \eqref{bound-p} holds. This bound together with the growth condition on $b$ gives

\begin{eqnarray*}
 I_4&=&\int_{\R}\int_0^t\int_{\R}\E b(u(s,\,y)) p(2t-s,\,x-y)\,\d y\,\d s p(t,\,x)\,\d x\\
&\geq&c_5\int_{\R}\int_{t_0}^{\tilde{c}_\alpha t}\left(\frac{s}{2t-s}\right)^{1/\alpha}\int_{\R} \E |u(s,\,y)|^{1+\gamma} p(s,\,y)\,\d y\,\d s p(t,\,x)\,\d x.\\
\end{eqnarray*}
We now apply Jensen's inequality twice to obtain 
\begin{eqnarray*}
I_4&\geq&c_5\int_{t_0}^{\tilde{c}_\alpha t}\left(\frac{s}{2t-s}\right)^{1/\alpha}\left(\int_{\R} \E |u(s,\,y)|  p(s,\,y)\,\d y\right)^{1+\gamma}\,\d s\\
&\geq&c_5\int_{t_0}^{\tilde{c}_\alpha t}\left(\frac{s}{2t}\right)^{1/\alpha}G(s)^{1+\gamma}\,\d s.\\
\end{eqnarray*}

Combining the bounds on $I_3$ and $I_4$, we obtain
\begin{equation*}
 G(t)\geq \kappa +c_5\int_{t_0}^{\tilde{c}_\alpha t}\frac{(G(s)s^{1/\alpha})^{1+\gamma}}{2t^{1/\alpha}s^{\gamma/\alpha}}\,\d s,
\end{equation*}
which in turn yields

\begin{equation*}
 G(t)t^{1/\alpha}\geq \kappa t^{1/\alpha} +c_5\int_{t_0}^{\tilde{c}_\alpha t}\frac{(G(s)s^{1/\alpha})^{1+\gamma}}{2 s^{\gamma/\alpha}}\,\d s\quad \text{for}\quad t\geq t_0.
\end{equation*}

We set $Y(t):=G(t)t^{1/\alpha}$ and look at the following ODE;  $Y'(t)t^{\gamma/\alpha}=c_6 Y(t)^{1+\gamma}$ with $Y(t_0)=\kappa t_0^{1/\alpha}$.  Now if $\gamma/\alpha\leq1$, then $Y(t)$ blows up in finite time, whenever the initial condition satisfies \eqref{ini}. But if $\gamma/\alpha>1$, then $Y(t)$ blows up whenever $\kappa$ is large enough. Since $G(t)$ blows up whenever $Y(t)$ blows up, the second part of the theorem is proved.
\end{proof}

We end this section with the following corollary whose proof is obtained by setting $b(u)=0$ in Theorem \ref{theorem1}. 

\begin{corollary}
\label{cor}
Consider the following equation

\begin{equation}\label{eq12}\begin{split}
\partial_t u(t,\,x)&=-(-\Delta)^{\alpha/2} u(t,\,x)+\sigma(u(t,\,x))\dot{w}(t,\,x)\\
u(0\,,x)&=u_0(x),\\
\end{split}
\end{equation}

Suppose that condition (1.1) is satisfied, then the second moment of the solution to \eqref{eq12} blows up in finite time as long as $\kappa$ is large enough.
\end{corollary}

\section{Remarks and Extensions.} 

In \cite{Su}, the author considered the following deterministic equation,

\begin{equation*}\begin{split}
\partial_t u(t,\,x)&=-(-\Delta)^{\alpha/2} u(t,\,x)+f(u(t,\,x))\\
u(0\,,x)&=u_0(x),\\
\end{split}
\end{equation*}
 and explored non-existence of global solutions under certain suitable conditions on $f$ and the initial data $u_0(x)$.  More precisely, it was shown that if $f(x)\geq c|x|^{1+\beta}$ with $0<\beta\leq \alpha$, then the solution to the above equation blows up for any positive initial condition.   We can randomise the above equation by multiplying the forcing term by a white noise and hence end up with \eqref{eq1} but with $b(u)=0$. We can then make the following interesting observation. According to Corollary \ref{cor} finite time blow up occurs for large enough initial data even for $\beta \leq \alpha$. However, in the deterministic case, as shown in \cite{Su}, if $\beta \leq \alpha$, there is blow up for any initial data.  Thus in this case, one might say that the presence of noise might {\it hinder} blow up. That is there is a range of initial data, for which there is definite blow up in the deterministic case, but there {\it might not} be blow up in the stochastic case.  It is also interesting to note that that the same mechanism has been reported in the literature for the stochastic wave equation, when forced by a Wiener field \cite{Gao11}.

We end this paper with a couple of possible extensions.  Our main result was proved under the assumption that $d=1$. It would be interesting to prove a similar result for $d>1$ but with a noise which has some spatial covariance.  One would thus need to find suitable conditions on the spatial covariance as well. This might not be an easy task as has been shown in \cite{FK1} where such equations were studied.  One should also note that under Condition \ref{B}, blow up mainly occurs due to the additive non-linearity and the theorem is proved by looking at the expectation of the solution. Hence in this case, one can easily extend the result to higher dimensions.

Another interesting question is to consider blow up or global existence of \eqref{eq1}, when $b(u)$ plays the role of a damping term, instead of a source term. In this case there will be competition between $b(u)$ and $\sigma(u)\dot{w}(t,\,x)$, and possible blow up will ultimately depend on which term dominates.

\bibliographystyle{amsplain}%
\bibliography{FoonParsh}

\vskip.4cm
\begin{small}
\noindent\textbf{Mohammud Foondun} \& \textbf{Rana D. Parshad}\\
\noindent School  of Mathematics, Loughborough University,
	LE11 3TU, UK \& Mathematics and Computer Science and Engineering division , KAUST, Saudi Arabia\\
\noindent\emph{Emails:} \texttt{m.i.foondun@lboro.ac.uk} \&
	\texttt{Rana.Parshad@kaust.edu.sa}\\
\end{small}
\end{document}